\documentclass[12pt,a4paper]{article}
\usepackage{preamble}
\usepackage{graphicx}
\usepackage{cleveref}
\usepackage{authblk}
\usepackage{enumitem}
\usepackage{blindtext}
\tikzset{edgee/.style = {->,> = latex'}}
\newcolumntype{P}[1]{>{\centering\arraybackslash}p{#1}}
\newcolumntype{M}[1]{>{\centering\arraybackslash}m{#1}}
\title{\textbf{Symmetric Union Closed Families }}
\author{Nived J M}
\date{} % 
\affil{\footnotesize M.Sc Mathematics, 
Indian Institute of Technology Hyderabad, India\\
Email: nivedjm.res@gmail.com
}

\begin{document}

\maketitle

\begin{abstract}
    We demonstrate that when a graph exhibits a specific type of symmetry, it satisfies the Union Closed Conjecture(UCC). Additionally, we show that certain graph classes, such as Cylindrical Grid Graphs and Torus Grid Graphs also satisfy the conjecture. We prove the known result that the union closed family generated by cyclic translates of a fixed set satisfies the UCC, offering a simpler proof via symmetry arguments. Later, we show that the union closed family generated by the family obtained through cyclically shifting elements from selected translates also satisfies the conjecture.
   \end{abstract}


\section{Introduction}\label{sec1}
The Union Closed Conjecture, also known as Frankl's Conjecture, was proposed by Peter Frankl in 1979. A family of sets is said to be \textit{union closed} if the union of any two of its member sets is also a member of the family. The conjecture asserts that for any finite union closed family of sets, excluding the family consisting the empty set alone, there exists an element that is present in at least half of the sets in the family. This element is referred to as an \textit{abundant} element. The set of all elements of the family $\mathcal{F}$ is the universe of $\mathcal{F}$, denoted by $\textit{U}(\mathcal{F})$.

 The Union Closed Conjecture has been proven in specific cases, such as when \( |\mathcal{F}| \geq \frac{2}{3} 2^{|U(\mathcal{F})| }\) \cite{2/3}, \( |\mathcal{F}|\leq 2|U(\mathcal{F})| \) \cite{2m}, and for smaller sets like \( |U(\mathcal{F})| \leq 12 \) \cite{12el} or \( |\mathcal{F}| \leq 50 \) \cite{50b,50c}. More recently Gilmer\cite{gil} made significant progress on the conjecture by providing the first constant lower bound for the conjecture. Despite these advancements, the conjecture remains open with various cases still unsolved.

The conjecture has been explored in multiple contexts, one of which is the \textit{graph formulation}\cite{11}. A vertex subset is called a \textit{maximal stable set} if no two vertices in the subset are adjacent and no additional vertex can be added to it without violating this condition. The graph formulation of the Union Closed Conjecture states that in any bipartite graph with at least one edge, there exists a vertex in each bipartite class that lies in at most half of the maximal stable sets. This vertex is known as \textit{rare}. 

We now turn our attention to the union closed family generated by a family $\mathcal{F}$, denoted by $\langle \mathcal{F} \rangle$, which consists of all possible unions of the sets in $\mathcal{F}$, including the empty set $\emptyset$. The \textit{incidence graph} of a family $\mathcal{F}$ is a bipartite graph $G$ with vertex set $V(G) = \mathcal{F} \cup U(\mathcal{F})$, and edge set $E(G) = \{Sx : S \in \mathcal{F}, x \in U(\mathcal{F}), x \in S \}$. For a bipartite graph $G$ with bipartite classes $X$ and $Y$, the \textit{incidence family} of $X$ is defined as $\mathcal{F}^{\scriptscriptstyle X} = \{N(y) : y \in Y\}$, where $N(y)$ is the set of neighbors of vertex $y$. The following proposition connects the concept of \textit{rare} vertices in the graph formulation to the concept of \textit{abundant} elements in the union closed family:

\begin{proposition}\label{prop31}\cite{me}
    Let $G$ be a bipartite graph with vertex partition $X \cup Y$. A vertex $x \in X$ is \textit{rare} if and only if it is \textit{abundant} in $\langle \mathcal{F}^{\scriptscriptstyle X} \rangle$.
\end{proposition}

In Section \ref{sec2}, we discuss how symmetry plays a role in the Union Closed Conjecture and provide a proof of the conjecture for Cylindrical and Torus Grid Graphs. These graphs are referenced in works such as \cite{to}, \cite{to1}, and \cite{to2}. In Section \ref{sec3}, we present a different proof for the UCC for families generated by translates of a fixed set\cite{da}. Moreover, we extend this result by showing that the conjecture also holds for families obtained by selecting certain sets within the family of translates of a fixed set, choosing an element from one of these sets, and applying cyclic position shifts to its translations across the selected sets under some conditions.

\section{Symmetry in UCC}\label{sec2}

To explore the UCC for graphs with specific symmetries, we introduce a theorem. For a graph $G$, a function $f \colon V(G) \to V(G)$ is called an automorphism if it is bijective and preserves the edge structure of the graph. Specifically, this means $(v_1, v_2) \in E(G)$ if and only if $(f(v_1), f(v_2)) \in E(G)$.

\begin{theorem}\label{thm1}
    Let $G$ be a bipartite graph with bipartition classes $X$ and $Y$. If there exists an automorphism $f$ on $G$ such that $f(X) = Y$ and $f(Y) = X$, then $G$ satisfies UCC.
\end{theorem}

\begin{proof}
    Note that, every bipartite graph contains at least one rare vertex, because having two adjacent vertices both abundant would lead to a contradiction. Without loss of generality, let $a \in X$ be the rare vertex. Since the vertex-edge structure is preserved by $f$, the image $f(a) \in Y$ is also rare.
\end{proof}

Theorem \ref{thm1} implies that for such a symmetry to hold, both bipartite classes must contain the same number of vertices. Now, we examine some specific graph classes to which this theorem applies.
\begin{proposition}
    The grid graph $G = P_m \times P_n$ (where $P_m$ denotes a path graph with $m$ vertices) satisfies the Union Closed Conjecture if either $n$ or $m$ (or both) are even.
\end{proposition}

An automorphism in the grid graph can be observed by rotating around a horizontal or vertical axis through the center. However, this observation is unsurprising, as a neighbor of any degree-2 vertex (corner vertices in this case) is rare\cite{sin}. Notably, these degree-2 vertices are not all located within the same bipartite class, allowing this result to be concluded directly.

Let us now turn to a more intriguing case. A \textit{Cylindrical Grid Graph}, mathematically expressed as $G = C_m \times P_n$, is a variation of the grid graph where vertices are connected in a cylindrical pattern, with $C_m$ representing a cycle graph with $m$ vertices. This graph is bipartite when $m$ is even.

\begin{proposition}
    Bipartite cylindrical grid graphs satisfy the Union Closed Conjecture.
\end{proposition}

\begin{proof}
    Label the vertices of $G = C_m \times P_n$ by $(i, j)$, where $i \in [m] - 1$\footnote{By \( [m] \) we mean \( \{ 1, 2, \dots, m \} \)} and $j \in [n] - 1$. Note that the vertices $(i, j)$ and $(i', j')$ are adjacent if $i = i'$ and $|j - j'| = 1$, or if $j = j'$ and $|i - i'| \equiv 1 \pmod{m}$. The mapping $f(i, j) = ((i + 1) \mod m, j)$ is an automorphism that maps vertices to the opposite bipartite class. By Theorem \ref{thm1}, $G$ satisfies the Union Closed Conjecture.
\end{proof}

An extension of the cylindrical grid graph is the \textit{Torus Grid Graph}, represented mathematically as $G = C_m \times C_n$. This construction yields a fully periodic grid structure where each vertex has exactly four neighbors, and the graph is bipartite when both $m$ and $n$ are even.

\begin{proposition}
    Bipartite torus grid graphs satisfy the Union Closed Conjecture.
\end{proposition}

\begin{proof}
    Represent the vertices of $G = C_m \times C_n$ by $(i, j)$, where $i \in [m] - 1$ and $j \in [n] - 1$. Vertices $(i, j)$ and $(i', j')$ are adjacent if $i = i'$ and $|j - j'| = 1 \pmod{n}$, or if $j = j'$ and $|i - i'| = 1 \pmod{m}$. The automorphism $f(i, j) = ((i + 1) \mod m, j)$ shifts every vertex to the opposite bipartite class, so by Theorem \ref{thm1}, $G$ satisfies UCC.
\end{proof}

Additionally, several bipartite graph classes, such as Möbius ladder, hypercube graphs, and crown graphs, satisfy the Union Closed Conjecture, which can be shown by identifying suitable automorphisms.

\section{Some Interesting Applications}\label{sec3}

In 2021, James Aaronson, David Ellis, and Imre Leader\cite{da} proved that the Union Closed Conjecture holds for union closed families generated by the cyclic translates of any fixed set. Here, we present an alternative proof, leveraging symmetry arguments in the graph formulation, which further simplifies the approach. 

For a cyclic group \( (H, +) \) and a non-empty fixed set \( R \subseteq H \), the translation of \( R \) by \( g \in H \) is defined by \( g + R := \{ g + r : r \in R \} \). Let \( \mathcal{F} \) be the family of all cyclic translates of fixed set \( R \neq \emptyset \), \( R \subseteq \mathbb{Z}_{n} \), where \( n \in \mathbb{N} \), with \( \mathcal{F} \) containing \( k \in \mathbb{N} \) member sets. We fix any set \( A \in \mathcal{F} \) and note that the member sets of \( \mathcal{F} \) are given by \( A, A+1, \dots, A+(k-1) \).

\begin{proposition}\label{pro1}
     The union closed family generated by \( \mathcal{F} \) satisfies UCC.
\end{proposition}
\begin{proof}
To establish the structure of our proof, we utilize Theorem \ref{thm1} and Proposition \ref{prop31} as primary tools. Consider the incidence graph \( G(X, Y) \) associated with the family \( \mathcal{F} \), where \( X = U(\mathcal{F}) = \mathbb{Z}_{n} \) and \( Y = \mathcal{F} \). Here, \( |X| = n \) and \( |Y| = k \). Although an automorphism as specified in Theorem \ref{thm1} does not directly exist, we observe that \( k \mid n \), allowing us to introduce \( \frac{n}{k} \) copies of each vertex in \( Y \), thereby matching the cardinalities \( |X| = |Y| \). Adding these copies does not affect the conjecture’s validity, as the union closed family generated by the incidence graph of \( X \) remains unchanged. We denote these copies by \( (A + i)_{\scriptscriptstyle c} \), where \( (A + i)_{\scriptscriptstyle c} \) represents the \( c^{\text{th}} \) copy of \( (A + i) \). Consequently, the modified bipartite class is \( Y = \{ (A + i)_{\scriptscriptstyle c} : i \in [k] - 1, c \in [\frac{n}{k}] \} \).

Define a mapping \( f : V(G) \rightarrow V(G) \) such that for any vertex \( (A + i)_{\scriptscriptstyle c} \in Y \), we set \( f((A + i)_{\scriptscriptstyle c}) = ck - i \). For any vertex \( a \in X \), we define \( f(a) = (A - a_{\scriptscriptstyle r})_{a_{\scriptscriptstyle d} + 1} \), where \( a_{\scriptscriptstyle d} = \left\lfloor \frac{a}{k} \right\rfloor \) and \( a_{\scriptscriptstyle r} = a \mod k \). It is straightforward to verify that \( f \) is a bijection, satisfying \( f(X) = Y \) and \( f(Y) = X \).

To confirm that \( f \) is an automorphism, we need to show that \( a \in (A + i)_{\scriptscriptstyle c} \) whenever \( f((A + i)_{\scriptscriptstyle c}) \in f(a) \), which is \( ck - i \in (A - a_{\scriptscriptstyle r})_{a_{\scriptscriptstyle d} + 1} \). Observe that \( a_{\scriptscriptstyle r} = a - k a_{\scriptscriptstyle d} \), allowing us to rewrite this condition as \( ck - i \in (A - a + k a_{\scriptscriptstyle d}) \). This relation implies that \( a \in (A + i + k a_{\scriptscriptstyle d} - ck) = (A + i) \) and as a result, \(a\in (A + i)_{\scriptscriptstyle c} \). Therefore, we conclude that \( f \) preserves the adjacency structure, confirming that it is indeed an automorphism of \( G \).  Consequently, it follows from Theorem \ref{thm1} that \( G\), and hence \( \left\langle\mathcal{F}\right \rangle \), satisfies the UCC.
\end{proof}

Now let us go to another interesting result when \( \mathcal{F} \) contains exactly \( n \) member sets. To clarify the process, we represent these member sets as tuples\footnote{This notation is for clarity, while all standard set operations remain valid.}. Observe that all these sets have the same cardinality. Let \( (A + i)(j) \) denote the \( j^{\text{th}} \) element of the tuple \( A + i \). By fixing the positions of elements in the tuple \( A \), the positions of elements in the tuple \( A + i \) are preserved according to the relationship \( (A + i)(j) = A(j) + i \). This is the translate of \( A(j) \) in \( (A + i) \).

We fix a set of indices \( I  \subseteq[n]-1\) and a bijective mapping \( q: I \longrightarrow I \) associated with it. The set $I$ is said to be \textit{$r-$suitable index} if it satisfies: 
\begin{enumerate}
\setlength{\itemsep}{1pt}
    \item \( I=r-I \)\label{it1}
    \item \( \forall i \in I, \quad r - i = q(r - q(i)) \)\label{it2}
\end{enumerate}

Now, consider any element of \( A \); without loss of generality, let us take \( A(1) \). Define a mapping \( P_{\scriptscriptstyle I, q} \) on \( \mathcal{F} \) as follows:
\[
P_{\scriptscriptstyle I,q}(A+i) = 
\begin{cases*}
A+i & if $i \notin I$, \\
(A+i)^{\prime} & if $i \in I$.
\end{cases*}
\]
where \( (A + i)^{\prime} \) is defined such that \( (A + i)^{\prime}(b) = (A + i)(b) \) for \( 2 \leq b \leq |A| \), and \( (A + i)^{\prime}(1) = (A + q(i))(1) \). In essence, we select \( l \) member sets of \( \mathcal{F} \), choose an element from one of these sets, and apply a cyclic position shift to the translates of this element across the selected sets in a certain pattern.

\begin{theorem}\label{th}
    Let \( |\mathcal{F}| = n \) and \( I\subseteq [n]-1 \) be an \( r \)-suitable index with respect to the mapping $q$. Then for any \( A \in \mathcal{F} \), the union closed family generated by \( \mathcal{F^{\prime}} = \{ P_{\scriptscriptstyle I, q}(A + i) : i \in [n] - 1 \} \) satisfies the Union Closed Conjecture.
\end{theorem}

\begin{proof}
    Let \( G(X,Y) \) be the incidence graph of \( \mathcal{F^{\prime}} \), where \( X = U(\mathcal{F^{\prime}}) = \mathbb{Z}_{\scriptscriptstyle n} \) and \( Y = \mathcal{F^{\prime}} \). It is evident that \( |X| = |Y| \) as \( |\mathcal{F^{\prime}}| = n \). Define a mapping \( f \colon V(G) \longrightarrow V(G) \) such that for any vertex \( P_{\scriptscriptstyle I, q}(A+i) \) of \( Y \), we set \( f(P_{\scriptscriptstyle I, q}(A+i)) = r + A(1) - i \). For each vertex \( a \in X \), define \( f(a) = P_{\scriptscriptstyle I, q}(A+r+A(1)-a) \). It is straightforward to observe that \( f \) is a bijective mapping, with \( f(X) = Y \) and \( f(Y) = X \).

    To prove that \( f \) is an automorphism, we need to show that \( a \in P_{\scriptscriptstyle I, q}(A+i) \iff r + A(1) - i \in P_{\scriptscriptstyle I, q}(A + r + A(1) - a) \). Let us define \( J = ([n] - 1) \setminus I \) and reduce the problem into four cases.

    \noindent\textbf{Case 1: \( i \in I \) and \( r + A(1) - a \in J \):} 
    Since \( i \in I \), we have \( P_{\scriptscriptstyle I, q}(A+i) = (A+i)^{\prime} \), whereas \( P_{\scriptscriptstyle I, q}(A + r + A(1) - a) = (A + r + A(1) - a) \) because \( r + A(1) - a \notin I\). Note that \( r + A(1) - i \in (A + r + A(1) - a) \iff a \in (A+i) \). This reduces to proving that \( a \in (A+i)^{\prime} \iff a \in (A+i) \). This is trivial except in two cases: when \( a = (A+i)(1) \) or \( a = (A+i)^{\prime}(1) \). These cases occur when \( a - A(1) = i \) or \( a - A(1) = q(i) \), respectively. By assumption, \( r + A(1) - a \in J \), which implies \( a - A(1) \in r - J \). Since \( I \) is an \( r \)-suitable index, and by property \ref{it1}, we have \( I \cap (r - J) = \emptyset \). Thus, both cases are ruled out as impossible.

    \noindent\textbf{Case 2: \( i \in I \) and \( r + A(1) - a \in I \):} 
    The objective here is to demonstrate that \( a \in (A+i)^{\prime} \) whenever \( r + A(1) - i \in (A + r + A(1) - a)^{\prime} \). For every \( b \neq 1 \), we have \( r + A(1) - i = (A + r + A(1) - a)^{\prime}(b) \iff r + A(1) - i = A(b) + r + A(1) - a \iff a = A(b) + i \iff a = (A + i)^{\prime}(b) \). For the first element, we have \( r + A(1) - i = (A + r + A(1) - a)^{\prime}(1) \iff r + A(1) - i = A(1) + q(r + A(1) - a) \iff r - i = q(r + A(1) - a) \). By property \ref{it2}, we obtain \( q(r - q(i)) = q(r + A(1) - a) \iff a = A(1) + q(i) = (A + i)^{\prime}(1) \).

    \noindent\textbf{Case 3: \( i \in J \) and \( r + A(1) - a \in I \):} 
    In this case, we need to prove that \( a \in (A+i) \) whenever \( r + A(1) - i \in (A + r + A(1) - a)^{\prime} \). This is similar to Case 1 and is trivial except when \( a = A(1) + i \) or \( r + A(1) - i = A(1) + q(r + A(1) - a) \). Since \( i \in J \), we have \( r - i \in r - J \), but by definition, \( q(r + A(1) - a) \in I \). Additionally, \( a - A(1) \in r - I= I \) when \( i \in J \), which makes these cases impossible.

    \noindent\textbf{Case 4: \( i \in J \) and \( r + A(1) - a \in J \):} 
    This case is straightforward as we need to prove that \( a \in A+i \iff r + A(1) - i \in (A + r + A(1) - a) \).
\end{proof}

\begin{proposition}
    Let \( R \subseteq \mathbb{Z}_{\scriptscriptstyle n} \) be a set that admits \( n \) distinct translations. Let \( \mathcal{F^{\prime}} \) denote the family derived by cyclically shifting the translations of \( a \in R \) by a fixed integer length across the sets \( A, A+1, A+2, \dots, A+(l-1) \) for an arbitrary integer \( l \in [n] \) and any translate \( A \) of \( R \). Then the family \( \left\langle \mathcal{F^{\prime}} \right\rangle \) satisfies the UCC.
\end{proposition}

\begin{proof}
    Define the index set \( I = \{0, 1, 2, \dots, l-1\} \) and introduce the mapping \( q \) such that \( q(i) = (i + m) \mod l \), where \( m \in [n] - 1 \) represents the fixed integer length of the cyclic shift. Observing that \( I \) forms an \((l-1)\)-suitable index with respect to this shift mapping allows us to apply Theorem \ref{th}. Consequently, the proposition follows as an immediate result.
\end{proof}

\begin{example}
   
Consider the family of translates of the set \( \{1, 2, 4, 7\} \) within the cyclic group \( \mathbb{Z}_7 \):

$\mathcal{F} = \{\{1, 2, 4, 7\}, \{2, 3, 5, 1\}, \{3, 4, 6, 2\}, \{4, 5, 7, 3\}, \{5, 6, 1, 4\}, \{6, 7, 2, 5\}, \{7, 1, 3, 6\}\}$.

Let \( \mathcal{F}^{\prime} \) denote the family obtained by cyclically shifting (by length 1) the translations of \( 1 \) (1, 2, and 3) across the first three sets of \( \mathcal{F} \), specifically sets \( \{1, 2, 4, 7\} \), \( \{2, 3, 5, 1\} \), and \( \{3, 4, 6, 2\} \). The resulting family \( \mathcal{F}^{\prime} \) is then:

$\mathcal{F}^{\prime} = \{\{2, 4, 7\}, \{3, 5, 1\}, \{1, 4, 6, 2\}, \{4, 5, 7, 3\}, \{5, 6, 1, 4\}, \{6, 7, 2, 5\}, \{7, 1, 3, 6\}\}$.

This modified family \( \mathcal{F}^{\prime} \) satisfies UCC according to the above proposition. Notably, \( \mathcal{F}^{\prime} \) includes member sets of varying cardinalities (3 and 4), illustrating that the theorem effectively verifies the UCC even in non-trivial configurations that are typically challenging to confirm with standard techniques.
\end{example}

\section{Acknowledgment}
I would like to thank my advisor, Prof. Rogers Mathew, for introducing the Union Closed Conjecture to me.

\bibliographystyle{abbrv}
\bibliography{references}

\end{document}